\documentclass[final,1p,times,authoryear]{elsarticle}
\usepackage{amsmath}
\usepackage{amsthm}
\usepackage{amssymb}
\usepackage{amsfonts}

\newtheorem{theorem}{Theorem}[section]
\newtheorem{lemma}[theorem]{Lemma}
\newtheorem{corollary}[theorem]{Corollary}

\newtheorem{proposition}[theorem]{Proposition}
\newtheorem{definition}[theorem]{Definition}

\newtheorem{example}[theorem]{Example}

\def\D{\displaystyle}

\DeclareMathOperator{\ord}{ord} \DeclareMathOperator{\Card}{Card}

\DeclareMathOperator{\trdeg}{tr.deg} \DeclareMathOperator{\rk}{rk}
\DeclareMathOperator{\lcm}{lcm}\DeclareMathOperator{\Ker}{Ker}

\begin{document}

\begin{frontmatter}

\title{Bivariate Kolchin-type dimension polynomials of non-reflexive prime difference-differential ideals. The case of one translation}

\author{Alexander Levin}
\address{The Catholic University of America\\ Washington, D. C.
20064} \ead{levin@cua.edu}
\ead[url]{https://sites.google.com/a/cua.edu/levin}

\begin{abstract}
We use the method of characteristic sets with respect to two term
orderings to prove the existence and obtain a method of computation
of a bivariate Kolchin-type dimension polynomial associated with a
non-reflexive difference-differential ideal in the algebra of
difference-differential polynomials with several basic derivations
and one translation. In particular, we obtain a new proof and a
method of computation of the dimension polynomial of a non-reflexive
prime difference ideal in the algebra of difference polynomials over
an ordinary difference field. As a consequence, it is shown that the
reflexive closure of a prime difference polynomial ideal is the
inverse image of this ideal under a power of the basic translation.
We also discuss applications of our results to the analysis of
systems of algebraic difference-differential equations.
\end{abstract}


\end{frontmatter}

\section{Introduction}
\setcounter{equation}{0}
\renewcommand{\theequation}{\thesection.\arabic{equation}}

The role of dimension polynomials in differential and difference
algebra is similar to the role of Hilbert polynomials in commutative
algebra and algebraic geometry. An important feature of such
polynomials is that they describe in exact terms the freedom degree
of a continuous or discrete dynamic system as well as the number of
arbitrary constants in the general solution of a system of partial
algebraic differential or difference equations. Furthermore, a
dimension polynomial associated with a finite system of generators
of a differential, difference or difference-differential field
extension carries certain birational invariants, numbers that do not
change when we switch to another system of generators. These
invariants, which characterize the extension, are closely connected
with some other its important characteristics such as the
differential (respectively, difference or difference-differential)
transcendence degree, type and typical dimension of the extension.
Finally, properties of dimension polynomials associated with prime
differential (respectively, difference or difference-differential)
ideals provide a powerful tool in the dimension theory of the
corresponding rings (see, for example, \citep{Jo69}, \citep[Ch.
7]{KLMP99}, and \citep{LM95}).

Dimension polynomials in differential and difference algebra have
been extensively studied since 1960s when E. Kolchin \citep{Ko64}
introduced a concept of a differential dimension polynomial
associated with a finitely generated differential field extension
(see also \citep[Ch. II, Sect. 12]{Ko73}). The corresponding
characteristics of difference and inversive difference field
extensions were introduced in \citep{Le78} and \citep{Le80}.

In this paper we adjust a generalization of the Ritt-Kolchin method
of characteristic sets developed in \citep{Le13} to the case of
(non-inversive) difference-differential polynomials with one
translation and apply this method to prove the existence and outline
a method of computation of a bivariate dimension polynomial
associated with a non-reflexive difference-differential polynomial
ideal. Our main result (Theorem 4.2) can be viewed as an essential
generalization (in the case of one translation) of the existing
theorems on bivariate dimension polynomials of
difference-differential and difference field extensions, see
\citep[Theorem 5.4]{Le00} and \citep[Theorems 4.2.16 and
4.2.17]{Le08}. The latter theorems deal with extensions that arise
from factor rings of difference-differential (or difference)
polynomial rings by reflexive difference-differential (respectively,
difference) prime ideals. Our paper extends these results to the
case when the prime ideals are not necessarily reflexive, so the
induced translations of the factor rings are not necessarily
injective. We also extend the results of \citep{Le18}; in the last
section we prove that the reflexive closure of a prime difference
polynomial ideal is the inverse image of a power of the basic
translation (see Corollary 4.4), generalize this result to the case
of several translations and include a discussion of the obstacles
for extending the main result of the paper to this case.
Furthermore, it is shown that the ring of difference-differential
polynomials over a difference-differential field with one
translation satisfies the ascending chain conditions (ACC) for prime
difference-differential ideals (Corollary 4.5). Note that rings of
differential (respectively, difference) polynomials over a
differential (respectively, difference) fields do not satisfy the
ACC for differential (respectively, difference) ideals, see
\citep[Sect. 11]{Ritt32} and \citep[Ch. 2, Example 3]{Cohn65}. At
the same time, such rings, as well as rings of
difference-differential polynomials, satisfy the ACC for perfect
ideals. (In the most general case of perfect difference-differential
polynomial ideals, the corresponding result is proven in
\citep{Cohn70}.) Note that perfect difference (and
difference-differential) ideals are reflexive. It is not known, in
general, what is the smallest class of difference (or
difference-differential) non-reflexive ideals that satisfies the
ACC. The most recent results of this kind in the difference case
were obtained in \citep{Le15} (this work shows that non-reflexive
well-mixed difference polynomial ideals do not satisfy ACC) and
\citep{Wa18} where the author proved that non-reflexive well-mixed
radical difference polynomial ideals generated by binomials satisfy
the ACC. We also discuss the relationship between the obtained
difference-differential dimension polynomial and the concept of
strength of a system of algebraic difference-differential equations
in the sense of A. Einstein. This concept was introduced in
\citep{Ei53} as a qualitative characteristic of a system of partial
differential equations. The corresponding notion for systems of
partial difference equations was introduced in \citep[Sect.
6.4]{KLMP99}; a detailed discussion of this concept in the
difference case can be found in \citep[Sect. 7.7]{Le08}.
Furthermore, as a consequence of our main result, we obtain a new
proof and a method of computation of the dimension polynomial of a
non-reflexive prime difference ideal in the algebra of difference
polynomials over an ordinary difference field. The existence of such
a polynomial was first established in \citep[Sect. 4.4]{Hr12}, an
alternative proof was obtained in \citep[Sect. 5.1]{Wi13}. However,
these proofs are not constructive, while our approach leads to an
algorithm for computing dimension polynomials.

\section{Preliminaries}
\setcounter{equation}{0}
\renewcommand{\theequation}{\thesection.\arabic{equation}}
In this section we present some basic concepts and results used in
the rest of the paper.

\smallskip

Throughout the paper $\mathbb{Z}$, $\mathbb{N}$, and $\mathbb{R}$
denote the sets of all integers, all non-negative integers and all
real numbers, respectively. For any positive integer $p$, we set
$\mathbb{N}_{p} = \{1, \dots, p\}$. By a ring we always mean an
associative ring with unity. Every ring homomorphism is unitary
(maps unity onto unity), every subring of a ring contains the unity
of the ring, and every algebra over a commutative ring is unitary.
Every field considered below is supposed to have zero
characteristic.

\medskip

If $B = A_{1}\times\dots\times A_{k}$ is a Cartesian product of $k$
ordered sets with orders $\leq_{1},\dots, \leq_{k}$, respectively
($k\in \mathbb{N}$, $k\geq 1$), then by the product order on the set
$B$ we mean a partial order $\leq_{P}$ such that $(a_{1},\dots,
a_{k})\leq_{P}(a'_{1},\dots, a'_{k})$ if and only if
$a_{i}\leq_{i}a'_{i}$ for $i=1,\dots, k$. In particular, if $a =
(a_{1},\dots, a_{k}), \,a'=(a'_{1},\dots, a'_{k})\in
\mathbb{N}^{k}$, then $a\leq_{P}a'$ if and only if $a_{i}\leq
a'_{i}$ for $i=1,\dots, k$. We write $a <_{P}a'$ if $a\leq_{P}a'$
and $a\neq a'$.

\medskip

The proof of the following statement can be found in \citep[Ch. 0,
Lemma 15]{Ko73}.
\begin{lemma}
Let $A$ be an infinite subset of
$\mathbb{N}^{m}\times{\mathbb{N}_{n}}$ ($m,n\in\mathbb{N}$, $n\geq
1$). Then there exists an infinite sequence of elements of $A$,
strictly increasing relative to the product order, in which every
element has the same projection on $\mathbb{N}_{n}$.
\end{lemma}

\begin{center}
NUMERICAL POLYNOMIALS
\end{center}
\begin{definition}
A polynomial $f(t_{1}, \dots,t_{p})$ in $p$ variables $t_{1},\dots,
t_{p}$ ($p\in\mathbb{N},\, p\geq 1$) with rational coefficients is
called {\em numerical} if $f(t_{1},\dots, t_{p})\in\mathbb{Z}$ for
all sufficiently large $p$-tuples $(t_{1}, \dots,
t_{p})\in\mathbb{Z}^{p}$ (that is, there exist integers
$s_{1},\dots, s_{p}$ such that $f(r_{1},\dots, r_{p})\in\mathbb{Z}$
whenever $(r_{1},\dots, r_{p})\in\mathbb{Z}$ and $r_{i}\geq s_{i}$
for all $i = 1,\dots, p$.)
\end{definition}

Obviously, every polynomial with integer coefficients is numerical.
As an example of a numerical polynomial in $p$ variables with
non-integer coefficients ($p\in\mathbb{N}, p\geq 1$) one can
consider a polynomial $\D\prod_{i=1}^{p}{t_{i}\choose m_{i}}$ \,
where $m_{1},\dots, m_{p}\in\mathbb{N}$. (As usual, $\D{t\choose k}$
($k\in\mathbb{Z},\, k\geq 1$) denotes the polynomial
$\D\frac{t(t-1)\dots (t-k+1)}{k!}$ in one variable $t$,
$\D{t\choose0} = 1$, and $\D{t\choose k} = 0$ if $k < 0$.)

\medskip

The following theorem proved in \citep[Ch.2]{KLMP99} gives the
``canonical'' representation of a numerical polynomial. It also
shows that the words ``for all sufficiently large'' in Definition
2.2 can be replaced with ``all'', that is, a numerical polynomial
takes integer values for all integer values of its arguments.

\begin{theorem}
Let $f(t_{1},\dots, t_{p})$ be a numerical polynomial in $p$
variables and let $\deg_{t_{i}}\, f = m_{i}$ ($m_{1},\dots,
m_{p}\in\mathbb{N}$). Then $f(t_{1},\dots, t_{p})$ can be
represented as
\begin{equation}
f(t_{1},\dots t_{p}) =\D\sum_{i_{1}=0}^{m_{1}}\dots
\D\sum_{i_{p}=0}^{m_{p}} {a_{i_{1}\dots i_{p}}}{t_{1}+i_{1}\choose
i_{1}}\dots{t_{p}+i_{p} \choose i_{p}}
\end{equation}
with uniquely defined integer coefficients $a_{i_{1}\dots i_{p}}$.
\end{theorem}

In the rest of the section we deal with subsets of
$\mathbb{N}^{m+1}$ ($m$ is a positive integer) treated as a
Cartesian product $\mathbb{N}^{m}\times\mathbb{N}$ (so that the last
coordinate has a special meaning). If $a = (a_{1},\dots,
a_{m+1})\in\mathbb{N}^{m+1}$, we set $\ord_{1}a =
\D\sum_{i=1}^{m}a_{i}$ and $\ord_{2}a = a_{m+1}.$ Furthermore, we
treat $\mathbb{N}^{m+1}$ as a partially ordered set with respect to
the product order $\leq_{P}$.

\medskip

If $A\subseteq \mathbb{N}^{m+1}$, then $V_{A}$ will denote the set
of all elements $v\in\mathbb{N}^{m+1}$ such that there is no $a\in
A$ with $a\leq_{P}v$. Clearly, $v=(v_{1}, \dots , v_{m+1})\in V_{A}$
if and only if for any element $(a_{1},\dots , a_{m+1})\in A$, there
exists $i\in\mathbb{N}, 1\leq i\leq m+1$, such that $a_{i}
> v_{i}$. Furthermore, for any $r, s\in\mathbb{N}$, we set
$$A(r, s) = \{x = (x_{1},\dots, x_{m+1})\in A\,|\,\ord_{1}x\leq
r, \ord_{2}x\leq s\}.$$

The following theorem is a direct consequence of the corresponding
statement proved in \citep[Chapter 2]{KLMP99}; it generalizes the
well-known Kolchin's result on univariate numerical polynomials
associated with subsets of $\mathbb{N}^{m}$ (see \citep[Ch. 0, Lemma
17]{Ko73}).

\begin{theorem}
Let $A$ be a subset of $\mathbb{N}^{m+1}$. Then there exists a
numerical polynomial $\omega_{A}(t_{1}, t_{2})$ with the following
properties:

\smallskip

{\em (i)} \,  $\omega_{A}(r, s) = \Card\,V_{A}(r, s)$ for all
sufficiently large $(r, s)\in\mathbb{N}^{2}$. (As in Definition 2.2,
it means that there exist $r_{0}, s_{0}\in\mathbb{N}$ such that the
equality holds for all integers $r\geq r_{0}$, $s\geq s_{0}$; as
usual, $\Card M$ denotes the number of elements of a finite set
$M$).

\smallskip

{\em (ii)} \, $\deg_{t_{1}}\omega_{A}\leq m$ and
$\deg_{t_{2}}\omega_{A}\leq 1$ (so the total degree $\deg\omega_{A}$
of the polynomial does not exceed $m+1$).

{\em (iii)} \, $\deg\,\omega_{A} = m+1$ if and only if
$A=\emptyset$. In this case\, $\omega_{A}(t_{1}, t_{2}) =
\D{{t_{1}+m}\choose m}(t_{2}+1)$.

\smallskip

{\em (iv)} \, $\omega_{A}$ is a zero polynomial if and only if
$(0,\dots, 0)\in A$.
\end{theorem}

\begin{definition}
The polynomial $\omega_{A}(t_{1}, t_{2})$ whose existence is stated
by Theorem 2.4 is called the dimension polynomial of the set
$A\subseteq\mathbb{N}^{m+1}$ associated with the orders $\ord_{1}$
and $\ord_{2}$.
\end{definition}

Note that if $A\subseteq\mathbb{N}^{m+1}$ and $A'$ is the set of all
minimal elements of $A$ with respect to the product order on
$\mathbb{N}^{m+1}$, then the set $A'$ is finite (it follows from
Lemma 2.1) and $\omega_{A}(t_{1}, t_{2}) = \omega_{A'}(t_{1},
t_{2})$. The following statement, which is a consequence of
\citep[Proposition 2.2.11]{KLMP99}, gives a closed-form formula for
the polynomial $\omega_{A}(t_{1}, t_{2})$ where $A$ is a finite
subset of $\mathbb{N}^{m+1}$ (therefore, it allows one to compute
the dimension polynomial of any subset of $\mathbb{N}^{m+1}$ after
finding the minimum points of this subset).

\begin{theorem}
Let $A = \{a_{1}, \dots, a_{q}\}$ be a finite subset of
$\mathbb{N}^{m+1}$. Let $a_{i} = (a_{i1}, \dots, a_{i,m+1})$ \,
($1\leq i\leq q$) and for any $l\in\mathbb{N}$, $0\leq l\leq q$, let
$\Gamma (l,q)$ denote the set of all $l$-element subsets of the set
$\{1,\dots, q\}$. Furthermore, for any $\gamma\in \Gamma (l,q)$, let
$\bar{a}_{\gamma j} = \max \{a_{ij} | i\in \gamma\}$ ($1\leq j\leq
m+1$) and $b_{\gamma} =\D\sum_{i=1}^{m}\bar{a}_{\gamma i}$. Then
\begin{equation} \omega_{A}(t_{1}, t_{2}) =
\D\sum_{l=0}^{q}(-1)^{l}\D\sum_{\gamma \in \Gamma (l,q)} {t_{1}+m -
b_{\gamma}\choose m}(t_{2}+1- \bar{a}_{\gamma, m+1}).
\end{equation}
\end{theorem}

\medskip

\begin{center}
BASIC NOTATION AND TERMINOLOGY ON DIFFERENCE-DIFFERENTIAL RINGS
\end{center}

\begin{definition} A commutative ring $R$ considered together with finite sets $\Delta =
\{\delta_{1},\dots, \delta_{m}\}$ and $\Sigma = \{\sigma_{1},\dots,
\sigma_{n}\}$ of derivations and injective endomorphisms of $R$,
respectively, such that any two mappings of the set
$\Delta\bigcup\Sigma$ commute, is called a difference-differential
ring.
\end{definition}
In what follows, we consider a special case when the set $\Sigma$
consists of a single endomorphism $\sigma$ called a {\em
translation}. Powers of $\sigma$ will be referred to as {\em
transforms}. The set $\Delta\bigcup\{\sigma\}$ is said to be the
{\em basic set} of the difference-differential ring $R$, which is
also called a $\Delta$-$\sigma$-ring. If $R$ is a field, it is
called a {\em difference-differential field} or a
$\Delta$-$\sigma$-field. We will often use prefix $\Delta$-$\sigma$-
instead of the adjective ''difference-differential''.

Let $T$ be the free commutative semigroup generated by the set
$\Delta\bigcup\{\sigma\}$, that is, the semigroup of all power
products $$\tau = \delta_{1}^{k_{1}}\dots
\delta_{m}^{k_{m}}\sigma^{l} \hspace{0.3in} (k_{i}, l\in
\mathbb{N}).$$ The numbers $\ord_{\Delta}\tau =
\D\sum_{i=1}^{m}k_{i}\,\,\, \text{and}\,\,\, \ord_{\sigma}\tau = l$
are called the {\em orders} of $\tau$ with respect to $\Delta$ and
$\sigma$, respectively.  For every $r, s\in \mathbb{N}$, we set
$$T(r, s) = \{\tau\in T\,|\,\ord_{\Delta}\tau\leq
r,\,\ord_{\sigma}\tau\leq s\}.$$ Furthermore, $\Theta$ will denote
the subsemigroup of $T$ generated by $\Delta$, so every element
$\tau\in T$ can be written as $\tau = \theta\sigma^{l}$ where
$\theta\in\Theta$, $l\in \mathbb{N}$. If $r\in\mathbb{N}$, we set
$\Theta(r) = \{\theta\in\Theta\,|\,\ord_{\Delta}\theta\leq r\}$.

\begin{definition}
A subring (ideal) $R_{0}$ of a $\Delta$-$\sigma$-ring $R$ is said to
be a difference-differential  (or $\Delta$-$\sigma$-) subring of $R$
(respectively, a difference-differential (or $\Delta$-$\sigma$-)
ideal of $R$) if $R_{0}$ is closed with respect to the action of any
operator of $\Delta\bigcup\sigma$. In this case the restriction of a
mapping from $\Delta\bigcup\sigma$ on $R_{0}$ is denoted by the same
symbol. If a prime ideal $P$ of $R$ is closed with respect to the
action of $\Delta\bigcup\sigma$, it is called a prime
difference-differential (or $\Delta$-$\sigma$-) ideal of $R$.
\end{definition}

\begin{definition}
If $R$ is a $\Delta$-$\sigma$-field and $R_{0}$ a subfield of $R$
which is also a $\Delta$-$\sigma$-subring of $R$, then  $R_{0}$ is
said to be a $\Delta$-$\sigma$-subfield of $R$; $R$, in turn, is
called a difference-differential (or $\Delta$-$\sigma$-) field
extension or a $\Delta$-$\sigma$-overfield of $R_{0}$. In this case
we also say that we have a $\Delta$-$\sigma$-field extension
$R/R_{0}$.
\end{definition}

If $R$ is a $\Delta$-$\sigma$-ring and $S\subseteq R$, then the
intersection of all $\Delta$-$\sigma$-ideals of $R$ containing the
set $S$ is, obviously, the smallest $\Delta$-$\sigma$-ideal of $R$
containing $S$. This ideal is denoted by $[S]$; as an ideal, it is
generated by all elements $\tau\eta$ where $\tau\in T$, $\eta\in S$.
(Here and below we frequently write $\tau\eta$ for $\tau(\eta)$
($\tau\in T$, $\eta\in R$).\,) If the set $S$ is finite, $S =
\{\eta_{1},\dots, \eta_{p}\}$, we say that the
$\Delta$-$\sigma$-ideal $I = [S]$ is finitely generated (in this
case we write $I = [\eta_{1},\dots, \eta_{p}]$) and call
$\eta_{1},\dots, \eta_{p}$ difference-differential (or
$\Delta$-$\sigma$-) generators of $I$.

\begin{definition}
A $\Delta$-$\sigma$-ideal $I$ of a $\Delta$-$\sigma$-ring $R$ is
called {\em reflexive} if the inclusion $\sigma^{k}(a)\in I$ ($k\in
\mathbb{N}$, $a\in R$) implies that $a\in I$.
\end{definition}

For any $\Delta$-$\sigma$-ideal $I$ of $R$, the set $I^{\ast} =
\{a\in R\,|\,\sigma^{k}(a)\in I$ for some $k\in \mathbb{N}\}$ is the
smallest reflexive $\Delta$-$\sigma$-ideal containing $I$; it is
called the {\em reflexive closure} of $I$ in $R$.

If $K_0$ is a $\Delta$-$\sigma$-subfield of a
$\Delta$-$\sigma$-field $K$ and $S\subseteq K$, then the
intersection of all $\Delta$-$\sigma$-subfields of $K$ containing
$K_0$ and $S$ is the unique $\Delta$-$\sigma$-subfield of $K$
containing $K_0$ and $S$ and contained in every
$\Delta$-$\sigma$-subfield of $K$ containing $K_0$ and $S$. It is
denoted by $K_{0}\langle S\rangle$. If $S$ is finite, $S =
\{\eta_{1},\dots,\eta_{n}\}$, then $K$ is said to be a finitely
generated $\Delta$-$\sigma$-extension of $K_{0}$ with the set of
$\Delta$-$\sigma$-generators $\{\eta_{1},\dots,\eta_{n}\}$. In this
case we write $K = K_{0}\langle \eta_{1},\dots,\eta_{n}\rangle$. As
a field, $K_{0}\langle \eta_{1},\dots,\eta_{n}\rangle$ coincides
with the field $K_0(\{\tau \eta_{i} | \tau\in T, 1\leq i\leq n\}$).

\begin{definition}
Let $R$ and $S$ be two difference-differential rings with the same
basic set $\Delta\bigcup\{\sigma\}$, so that elements of $\Delta$
and $\sigma$ act on each of the rings as derivations and an
endomorphism, respectively, and every two mapping of the set
$\Delta\bigcup\{\sigma\}$ commute.  Then a ring homomorphism
$\phi:R\longrightarrow S$ is called a difference-differential (or
$\Delta$-$\sigma$-) homomorphism if $\phi(\alpha a) = \alpha\phi(a)$
for any $\alpha\in \Delta\bigcup\{\sigma\}$, $a\in R$.
\end{definition}

If $\phi:R\longrightarrow S$ is a $\Delta$-$\sigma$-ring
homomorphism, then $\Ker\phi$ is a reflexive $\Delta$-$\sigma$-ideal
of $R$. Furthermore, if $J$ is a reflexive $\Delta$-$\sigma$-ideal
of $\Delta$-$\sigma$-ring $R$, then the factor ring $R/J$ has a
natural structure of a $\Delta$-$\sigma$-ring such that the
canonical epimorphism $R\rightarrow R/J$ is a
$\Delta$-$\sigma$-homomorphism.

If $K$ is a $\Delta$-$\sigma$-field and $Y =\{y_{1},\dots, y_{n}\}$
is a finite set of symbols, then one can consider a countable set of
symbols $TY = \{\tau y_{j}|\tau\in T, 1\leq j\leq n\}$ and the
polynomial ring $R = K[\{\tau y_{j}|\tau\in T,\, 1\leq j\leq n\}]$
in the set of indeterminates $TY$ over the field $K$. This
polynomial ring is naturally viewed as a $\Delta$-$\sigma$-ring
where $\alpha(\tau y_{j}) = (\alpha\tau)y_{j}$ for any
$\alpha\in\Delta\bigcup\{\sigma\}$, $\tau\in T$, $1\leq j\leq n$,
and the elements of $\Delta\bigcup\{\sigma\}$ act on the
coefficients of the polynomials of $R$ as they act on the field $K$.
The ring $R$ is called the {\em ring of difference-differential} (or
$\Delta$-$\sigma$-) {\em polynomials} in the set of
difference-differential ($\Delta$-$\sigma$-) indeterminates
$y_{1},\dots, y_{n}$ over $K$. This ring is denoted by
$K\{y_{1},\dots, y_{n}\}$ and its elements are called
difference-differential (or $\Delta$-$\sigma$-) polynomials.

\begin{definition}A $\Delta$-$\sigma$-ideal in the ring $K\{y_{1},\dots, y_{n}\}$
is called linear if it is generated (as a $\Delta$-$\sigma$-ideal)
by homogeneous linear $\Delta$-$\sigma$-polynomials (i. e.,
$\Delta$-$\sigma$-polynomials of the form
$\sum_{i=1}^{d}a_{i}\tau_{i}y_{k_{i}}$ where $a_{i}\in K$,
$\tau_{i}\in T$, $1\leq k_{i}\leq n$).
\end{definition}
\begin{definition}
Let $R$ be a $\Delta$-$\sigma$-ring and $\mathcal{U}$ a family of
elements of some $\Delta$-$\sigma$-overring of $R$. We say that
$\mathcal{U}$ is $\Delta$-$\sigma$-algebraically) dependent over
$R$, if the family $T\mathcal{U} = \{\tau u\,|\,\tau\in T,\, u\in
\mathcal{U}\}$ is algebraically dependent over $R$ (that is, there
exist elements $u_{1},\dots, u_{k}\in T\mathcal{U}$ and a nonzero
polynomial $f$ in $k$ variables with coefficients in $R$ such that
$f(u_{1},\dots, u_{k}) = 0$). Otherwise, the family $\mathcal{U}$ is
said to be $\Delta$-$\sigma$-{\em algebraically independent} over
$R$.
\end{definition}

\begin{definition}If $K$ is a $\Delta$-$\sigma$-field and $L$ a
$\Delta$-$\sigma$-field extension of $K$, then a set $B\subseteq L$
is said to be a $\Delta$-$\sigma$-transcendence basis of $L$ over
$K$ if $B$ is $\Delta$-$\sigma$-algebraically independent over $K$
and every element $a\in L$ is $\Delta$-$\sigma$-algebraic over
$K\langle B\rangle$ (it means that the set $\{\tau a\,|\,\tau\in
T\}$ is algebraically dependent over the field $K\langle B\rangle$).
\end{definition}

If $L$ is a finitely generated $\Delta$-$\sigma$-field extension of
$K$, then all $\Delta$-$\sigma$-transcendence bases of $L$ over $K$
are finite and have the same number of elements (see \citep[Theorem
3.5.38]{KLMP99}). This number is called the $\Delta$-$\sigma$-{\em
transcendence degree} of $L$ over $K$ (or the
$\Delta$-$\sigma$-transcendence degree of the extension $L/K$); it
is denoted by $\Delta$-$\sigma$-$\trdeg_{K}L$.

Let $K$ be a $\Delta$-$\sigma$-field $K$ and $L$ a finitely
generated $\Delta$-$\sigma$-extension of $K$ with a set of
$\Delta$-$\sigma$-generators $\eta = \{\eta_{1},\dots, \eta_{n}\}$,
$L = K\langle \eta_{1},\dots,\eta_{n}\rangle$. Then there exists a
natural $\Delta$-$\sigma$-homomorphism $\Phi_{\eta}$ of the ring of
$\Delta$-$\sigma$-polynomials $K\{y_{1},\dots, y_{n}\}$ onto the
$\Delta$-$\sigma$-subring $K\{\eta_{1},\dots,\eta_{n}\}$ of $L$ such
that $\Phi_{\eta}(a) = a$ for any $a\in K$ and $\Phi_{\eta}(y_{j}) =
\eta_{j}$ for $j = 1,\dots, n$. If $A$ is a
$\Delta$-$\sigma$-polynomial in $K\{y_{1},\dots, y_{n}\}$, then the
element $\Phi_{\eta}(A)$ is called the {\em value} of $A$ at $\eta$
and is denoted by $A(\eta)$. Obviously, the kernel $P$ of the
$\Delta$-$\sigma$-homomorphism $\Phi_{\eta}$ is a prime reflexive
$\Delta$-$\sigma$-ideal of $K\{y_{1},\dots, y_{n}\}$. This ideal is
called the {\em defining} ideal of $\eta$. If we consider the
quotient field $Q$ of the factor ring $\bar{R} = K\{y_{1},\dots,
y_{n}\}/P$ as a $\Delta$-$\sigma$-field (where
$\delta\left({\D\frac{u}{v}}\right) =
{\D\frac{v\delta(u)-u\delta(v)}{v^2}}$ and
$\sigma\left({\D\frac{u}{v}}\right) =
{\D\frac{\sigma(u)}{\sigma(v)}}$ for any $u, v\in \bar{R},\,
\delta\in\Delta$), then this quotient field is naturally
$\Delta$-$\sigma$-isomorphic to the field $L$. The
$\Delta$-$\sigma$-isomorphism of $Q$ onto $L$ is identical on $K$
and maps the canonical images of the
$\Delta$-$\sigma$-indeterminates $y_{1},\dots, y_{n}$ in
$K\{y_{1},\dots, y_{n}\}/P$ to the elements $\eta_{1},\dots,
\eta_{n}$, respectively.

\section{Reduction of $\Delta$-$\sigma$-polynomials. Characteristic sets}

Let $K$ be a difference-differential field with a basic set
$\Delta\bigcup\{\sigma\}$ ($\Delta = \{\delta_{1},\dots,
\delta_{m}\}$ is a set of derivations, $\sigma$ is an endomorphism
of $K$). Let $R = K\{y_{1},\dots, y_{n}\}$ be the ring of
$\Delta$-$\sigma$-polynomials in the set of
$\Delta$-$\sigma$-indeterminates $y_{1},\dots, y_{n}$ over $K$ and
let $TY$ denote the set of all elements $\tau y_{i}\in R$ ($\tau\in
T,\, 1\leq i\leq n$) called {\em terms}. If $u = \tau y_{i}\in TY$,
then the numbers $\ord_{\Delta}\tau$ and $\ord_{\sigma}\tau$ are
called the orders of the term $u$ with respect to $\Delta$ and
$\sigma$, respectively.

We will consider two total orders $<_{\Delta}$ and $<_{\sigma}$ on
the set of all terms $TY$ defined as follows:

If $u = \delta_{1}^{k_{1}}\dots \delta_{m}^{k_{m}}\sigma^{p}y_{i}$
and $v = \delta_{1}^{l_{1}}\dots \delta_{m}^{l_{m}}\sigma^{q}y_{j}$
($1\leq i, j\leq n$), then $u<_{\Delta} v$ (respectively,
$u<_{\sigma} v$) if the $(m+3)$-tuple $(\ord_{\Delta}u,
\ord_{\sigma}u, k_{1},\dots, k_{m}, i)$ is less than
$(\ord_{\Delta}v, \ord_{\sigma}v, l_{1},\dots, l_{m}, j)$
(respectively, $(\ord_{\sigma}u, \ord_{\Delta}u, k_{1},\dots, k_{m},
i)$ is less than the $(m+3)$-tuple $(\ord_{\sigma} v,
\ord_{\Delta}v, l_{1},\dots, l_{m}, j)$) with respect to the
lexicographic order on $\mathbb{N}^{m+3}$. We write $u\leq_{\Delta}
v$ if either $u <_{\Delta} v$ or $u=v$; the relation $\leq_{\sigma}$
is defined in the same way.

An element $\tau\in T$ is said to be divisible by an element
$\tau'\in T$ if $\tau = \tau''\tau'$ for some $\tau''\in T$. In this
case we write $\tau'\,|\,\tau$ and $\tau'' = \D\frac{\tau}{\tau'}$.

The least common multiple of elements $\tau_{1},\dots, \tau_{p}\in
T$, where $\tau_{i} = \delta_{1}^{k_{i1}}\dots
\delta_{m}^{k_{im}}\sigma^{l_{i}}$ ($1\leq i\leq p$) is defined as
$\tau = \delta_{1}^{d_{1}}\dots \delta_{m}^{d_{m}}\sigma^{l}$ with
$d_{j} = \max \{k_{1j},\dots, k_{pj}\}$  ($1\leq j\leq m$), $l =
\max\{l_{1},\dots, l_{p}\}$; it is denoted by $\lcm\{\tau_{1},\dots,
\tau_{p}\}$.

If $u = \tau_{1} y_{i},\, v = \tau_{2} y_{j}\in TY$, we say that $u$
divides $v$ and write $u\,|\,v$ if and only if $i=j$ and
$\tau_{1}\,|\,\tau_{2}$. In this case the ratio $\D\frac{v}{u}$ is
defined as ${\D\frac{\tau_{2}}{\tau_{1}}}$. If $u_{1} = \tau_{1}
y_{i}, \dots, u_{p} = \tau_{p} y_{i}$ are terms with the same
$\Delta$-$\sigma$-indeterminate $y_{i}$, then the least common
multiple of these terms, denoted by $\lcm(u_{1},\dots, u_{p})$, is
defined as $\lcm(\tau_{1},\dots, \tau_{p})y_{i}$.  If $i\neq j$,
then we set $\lcm(\tau_{1}y_{i}, \tau_{2}y_{j}) = 0$ ($\tau_{1},
\tau_{2}\in T$). The following statement is a consequence of Lemma
2.1.

\begin{lemma}
Let $S$ be any infinite set of terms in the ring $K\{y_{1},\dots,
y_{n}\}$. Then there exists an infinite sequence of terms $u_{1},\,
u_{2},\dots$ in $S$ such that $u_{k}\,|\,u_{k+1}$ for every $k=1, 2,
\dots $.
\end{lemma}

\begin{definition}If $A\in K\{y_{1},\dots, y_{n}\}\setminus K$, then the highest with
respect to the orderings  $<_{\Delta}$ and $<_{\sigma}$ terms that
appear in $A$ are called the $\Delta$-leader and the $\sigma$-leader
of $A$; they are denoted by $u_{A}$ and $v_{A}$, respectively. If
$A$ is written as a polynomial in one variable $v_{A}$, $A =
I_{d}(v_{A})^{d} + I_{d-1}(v_{A})^{d-1} + \dots + I_{0}$
($\Delta$-$\sigma$-polynomials $I_{d}, I_{d-1}, \dots, I_{0}$ do not
contain $v_{A}$), then $I_{d}$ is called the  leading coefficient of
$A$; the partial derivative $\partial A/\partial v_{A} =
dI_{d}(v_{A})^{d-1} + (d-1)I_{d-1}(v_{A})^{d-2} + \dots + I_{1}$ is
called the separant of $A$. The leading coefficient and the separant
of a $\Delta$-$\sigma$-polynomial $A$ are denoted by $I_{A}$ and
$S_{A}$, respectively.
\end{definition}
\begin{definition}
Let $A$ and $B$ be two $\Delta$-$\sigma$-polynomials in
$K\{y_{1},\dots, y_{n}\}$. We say that $A$ has lower rank than $B$
and write $\rk A < \rk B$ if either $A\in K$, $B\notin K$, or the
vector $(v_{A}, \deg_{v_{A}}A, \ord_{\Delta}u_{A})$ is less than
$(v_{B}, \deg_{v_{B}}B, \ord_{\Delta}u_{B})$ with respect to the
lexicographic order (where the terms $v_{A}$ and $v_{B}$ are
compared with respect to the order $<_{\sigma}$ and the other
coordinates are compared with respect to the natural order on
$\mathbb{N}$). If the two vectors are equal (or $A, B\in K$), we say
that the $\Delta$-$\sigma$-polynomials $A$ and $B$ are of the same
rank and write $\rk A = \rk B$.
\end{definition}

\begin{definition}
If $A, B\in K\{y_{1},\dots, y_{n}\}$, then $B$ is said to be {\bf
reduced} with respect to $A$ if

(i) $B$ does not contain terms $\tau v_{A}$ such that
$\ord_{\Delta}\tau > 0$ and $\ord_{\Delta}(\tau u_{A})\leq
\ord_{\Delta}u_{B}$.

\smallskip

(ii) If $B$ contains a term $\tau v_{A}$ where $\ord_{\Delta}\tau =
0$, then either $\ord_{\Delta} u_{B} < \ord_{\Delta}u_{A}$ or
$\ord_{\Delta} u_{A}\leq\ord_{\Delta}u_{B}$ and $\deg_{\tau v_{A}}B
<\deg_{v_{A}}A$.

\smallskip

If $B\in K\{y_{1},\dots, y_{n}\}$, then $B$ is said to be reduced
with respect to a set\, $\mathcal{A}\subseteq K\{y_{1},\dots,
y_{n}\}$ if $B$ is reduced with respect to every element of
$\mathcal{A}$.
\end{definition}
\noindent{\bf Remark.} It follows from the last definition that a
$\Delta$-$\sigma$-polynomial $B$ is not reduced with respect to a
$\Delta$-$\sigma$-polynomial $A$ ($A\notin K$) if either $B$
contains some term $\tau v_{A}$ such that $\ord_{\Delta}\tau > 0$
and $\ord_{\Delta}(\tau u_{A})\leq\ord_{\Delta}u_{B}$ or $B$
contains $\sigma^{i}v_{A}$ for some $i\in\mathbb{N}$ and in this
case $\ord_{\Delta}u_{A}\leq\ord_{\Delta}u_{B}$ and
$\deg_{v_{A}}A\leq\deg_{\sigma^{i}v_{A}}B$.

\begin{definition}
A set of $\Delta$-$\sigma$-polynomials $\mathcal{A}$ in the ring
$K\{y_{1},\dots, y_{n}\}$ is called {\bf autoreduced} if
$\mathcal{A}\bigcap K = \emptyset$ and every element of
$\mathcal{A}$ is reduced with respect to any other element of this
set.
\end{definition}

\begin{proposition}
Every autoreduced set in the ring of $\Delta$-$\sigma$-polynomials
is finite.
\end{proposition}

\begin{proof}
Suppose that $\mathcal{A}$ is an infinite autoreduced subset of
$K\{y_{1},\dots, y_{n}\}$. Then there is an infinite subset
$\mathcal{A}'$ of $\mathcal{A}$ such that all
$\Delta$-$\sigma$-polynomials in $\mathcal{A}'$ have distinct
$\sigma$-leaders. Indeed, otherwise there exists an infinite set
$\mathcal{A}_{1}\subseteq \mathcal{A}$ such that all
$\Delta$-$\sigma$-polynomials in $\mathcal{A}_{1}$ have the same
$\sigma$-leader $v$. Then the infinite set
$\{\ord_{\Delta}u_{A}\,|\, A\in \mathcal{A}_{1}\}$ contains a
nondecreasing infinite sequence $\ord_{\Delta}u_{A_{1}}\leq
\ord_{\Delta}u_{A_{2}} \leq\dots$. Since the sequence
$\{\deg_{v}A_{i} | i = 1, 2, \dots \}$ cannot be strictly
decreasing, there exists two indices $i$ and $j$ such that $i < j$
and $\deg_{v}A_{i} \leq \deg_{v}A_{j}$. We obtain that $A_{j}$ is
not reduced with respect to $A_{i}$ that contradicts the fact that
$\mathcal{A}$ is an autoreduced set.

Thus, we can assume that all leaders of our infinite autoreduced set
$\mathcal{A}$ are distinct. By Lemma 3.1, there exists an infinite
sequence $B_{1}, B_{2}, \dots $ \, of elements of $\mathcal{A}$ such
that $v_{B_{i}}\,|\,v_{B_{i+1}}$ for all $i = 1, 2, \dots $. (Also,
since the leaders of elements of our sequence are distinct,
${\D\frac{v_{B_{i+1}}}{v_{B_{i}}}}\neq 1$.)

Let $k_{i} = \ord_{\Delta }v_{B_{i}}$ and  $l_{i} =
\ord_{\Delta}u_{B_{i}}$. Since $u_{B_{i}}$ is the $\Delta$-leader of
$B_{i}$, $l_{i}\geq k_{i}$ ($i = 1, 2,\dots $), so that the infinite
set $\{l_{i}-k_{i}\,|\,i\in\mathbb{N},\, i\geq 1\}$ contains a
nondecreasing sequence $l_{i_{1}}-k_{i_{1}},\,
l_{i_{2}}-k_{i_{2}},\dots$. Then
$\ord_{\Delta}(\frac{v_{B_{i_{2}}}}{v_{B_{i_{1}}}} u_{B_{i_{1}}}) =
k_{i_{2}} - k_{i_{1}} + l_{i_{1}}\leq k_{i_{2}} + l_{i_{2}} -
k_{i_{2}} = l_{i_{2}} = \ord_{\Delta}u_{B_{i_{2}}}$. It follows that
$B_{i_{2}}$ contains a term $\tau v_{B_{i_{1}}} = v_{B_{i_{2}}}$
such that $\ord_{\Delta}\tau > 0$ and $\ord_{\Delta}(\tau
u_{B_{i_{1}}})\leq\ord_{\Delta}u_{B_{i_{2}}}$. Thus, the
$\Delta$-$\sigma$-polynomial $B_{i_{2}}$ is reduced with respect to
$B_{i_{1}}$ that contradicts the fact that $\mathcal{A}$ is an
autoreduced set.
\end{proof}

The proof of the following statement is similar to the proof of the
reduction theorem for difference-differential polynomials in the
case of classical autoreduced sets, see \citep{Cohn70}.

\begin{proposition} Let $\mathcal{A}$ be an
autoreduced set in the ring of $\Delta$-$\sigma$-polynomials
$K\{y_{1},\dots, y_{n}\}$ and let $B\in K\{y_{1},\dots, y_{n}\}$.
Then there exist $\Delta$-$\sigma$-polynomials $B_{0}$ and $H$ such
that $B_{0}$ is reduced with respect to $\mathcal{A}$, $\rk
B_{0}\leq \rk B$, $H$ is a finite product of transforms of separants
and leading coefficients of elements of $\mathcal{A}$, and $HB\equiv
B_{0} \,(mod [\mathcal{A}])$ (that is, $HB-B_{0}\in [\mathcal{A}]$).
\end{proposition}
With the notation of the last proposition, we say that the
$\Delta$-$\sigma$-polynomial $B$ {\em reduces to $B_{0}$} modulo
$\mathcal{A}$.

Throughout the rest of the paper, while considering an autoreduced
set $\mathcal{A} = \{A_{1},\dots, A_{p}\}$ in the ring
$K\{y_{1},\dots, y_{n}\}$ we always assume that its elements are
arranged in order of increasing rank,  $rk\,A_{1}< \dots <
rk\,A_{p}$.

\begin{definition}
If $\mathcal{A} = \{A_{1},\dots,A_{p}\}$, $\mathcal{B} =
\{B_{1},\dots, B_{q}\}$ are  two autoreduced sets of
$\Delta$-$\sigma$-polynomials $K\{y_{1},\dots, y_{n}\}$, we say that
$\mathcal{A}$ has lower rank than $\mathcal{B}$ if one of the
following two cases holds:

(1) There exists $k\in\mathbb{N}$ such that $k\leq \min \{p,q\}$,
$\rk A_{i}= \rk B_{i}$ for $i=1,\dots,k-1$ and  $\rk A_{k} < \rk
B_{k}$.

(2) $p > q$ and  $\rk A_{i} = \rk B_{i}$ for $i=1,\dots,q$.

If $p = q$ and $\rk A_{i} = \rk B_{i}$ for $i=1,\dots,p$, then
$\mathcal{A}$ is said to have the same rank as $\mathcal{B}$. In
this case we write $\rk\mathcal{A} = \rk\mathcal{B}$.
\end{definition}

Repeating the arguments of the proof of the corresponding result for
autoreduced sets of differential polynomials (see \citep[Ch. I,
Proposition 3]{Ko73}) we obtain the following statement.

\begin{proposition}
In every nonempty family of autoreduced sets of differential
polynomials there exists an autoreduced set of lowest rank.
\end{proposition}

This statement shows that if $J$ is a $\Delta$-$\sigma$-ideal (or
even a nonempty subset) of the ring of $\Delta$-$\sigma$-polynomials
$K\{y_{1},\dots, y_{n}\}$, then $J$ contains an autoreduced subset
of lowest rank.(Clearly, the set of all autoreduced subsets of $J$
is not empty: if $A\in J$, then $\{A\}$ is an autoreduced subset of
$J$.)

\begin{definition}
If $J$ is a nonempty subset (in particular, a
$\Delta$-$\sigma$-ideal) of the ring of
$\Delta$-$\sigma$-polynomials $K\{y_{1},\dots, y_{n}\}$, then an
autoreduced subset of $J$ of lowest rank is called a {\bf
characteristic set} of $J$.
\end{definition}

\begin{proposition}
Let $\mathcal{A} = \{A_{1}, \dots , A_{p}\}$ be a characteristic set
of a nonempty subset $J$ of the ring of
$\Delta$-$\sigma$-polynomials $R=K\{y_{1},\dots, y_{n}\}$.  Then an
element $B\in J$ is reduced with respect to $\mathcal{A}$ if and
only if $B = 0$. In particular, the separants and leading
coefficients of elements of $\mathcal{A}$ do not lie in $J$.
\end{proposition}

\begin{proof}
First of all, note that if $B\neq 0$ and $\rk B < \rk A_{1}$, then
$\{B\}$ is an autoreduced set and $\rk\{B\} < \rk\mathcal{A}$ that
contradicts the fact that $\mathcal{A}$ is a characteristic set of
$J$. Let $\rk B > \rk A_{1}$ and let $A_{1},\dots, A_{j}$ ($1\leq
j\leq p$) be all elements of $\mathcal{A}$ whose rank is lower that
the rank of $B$. Then the set $\mathcal{A}' = \{A_{1},\dots, A_{j},
B\}$ is autoreduced. Indeed, the $\Delta$-$\sigma$-polynomials
$A_{1},\dots, A_{j}$ are reduced with respect to each other and $B$
is reduced with respect to the set $\{A_{1},\dots, A_{j}\}$, since
$B$ is reduced with respect to $\mathcal{A}$. Furthermore, each
$A_{i}$ ($1\leq i\leq j$) is reduced with respect to $B$ because
$\rk A_{i} < \rk B$. By the choice of $B$, if $j < p$, then $\rk B <
\rk A_{j+1}$, so $\rk\mathcal{A}' < \rk\mathcal{A}$; if $j=p$, then
we still have the inequality $\rk\mathcal{A}' < \rk\mathcal{A}$ by
the second part of Definition 3.8. It follows that  $\mathcal{A}$ is
not a characteristic set of $J$, contrary to our assumption. Thus,
$B = 0$.
\end{proof}

\begin{definition}
Let $\mathcal{A} = \{A_{1},\dots, A_{p}\}$ be an autoreduced set in
the ring $K\{y_{1},\dots, y_{n}\}$ such that all
$\Delta$-$\sigma$-polynomials $A_{i}$ ($1\leq i\leq p$) are linear
and homogeneous. Then the set $\mathcal{A}$ is said to be coherent
if it satisfies the following two conditions.

(i)\, $\tau A_{i}$ reduces to zero modulo $\mathcal{A}$ for any
$\tau\in T, \,1\leq i\leq p$.

(ii)\, For every $A_{i},\, A_{j}\in\mathcal{A}$ , $1\leq i < j\leq
p$, let $w = \lcm\{v_{A_{i}}, v_{A_{j}}\}$ and let $\tau' =
{\D\frac{w}{v_{A_{i}}}}$, $\tau'' = {\D\frac{w}{v_{A_{j}}}}$. Then
the $\Delta$-$\sigma$-polynomial $(\tau''I_{A_{j}})(\tau'A_{i}) -
(\tau'I_{A_{i}})(\tau''A_{j})$ reduces to zero modulo $\mathcal{A}$.
\end{definition}

The proof of the following statement can be obtained by mimicking
the proof of the corresponding result for autoreduced sets of
difference polynomials, see \citep[Theorem 6.5.3]{KLMP99}.

\begin{proposition}
Every characteristic set of a linear $\Delta$-$\sigma$-ideal in the
ring of  $\Delta$-$\sigma$-polynomials $K\{y_{1},\dots, y_{n}\}$ is
a coherent autoreduced set. Conversely, if $\mathcal{A}$ is a
coherent autoreduced set in $K\{y_{1},\dots, y_{n}\}$ consisting of
linear homogeneous $\Delta$-$\sigma$-polynomials, then $\mathcal{A}$
is a characteristic set of the linear $\Delta$-$\sigma$-ideal
$[\mathcal{A}]$.
\end{proposition}

\section{Dimension polynomials. The main theorem}

Let $K$ be a $\Delta$-$\sigma$-field (as before $\Delta =
\{\delta_{1},\dots, \delta_{m}\}$ is a set of mutually commuting
derivations of $K$ and $\sigma$ is an endomorphism of $K$ that
commutes with every $\delta_{i}$). Let $R = K\{y_{1},\dots, y_{n}\}$
be the ring of $\Delta$-$\sigma$-polynomials over $K$ and $P$ a
prime $\Delta$-$\sigma$-ideal of $R$. Let $P^{\ast}$ denote the
reflexive closure of $P$ in $R$ (as we have mentioned, $P^{\ast}$ is
also a prime $\Delta$-$\sigma$-ideal of $R$) and for every $r, s\in
\mathbb{N}$, let $R_{rs} = K[\{\tau y_{i}\,|\,\tau\in T(r, s), 1\leq
i\leq n\}]$. In other words, $R_{rs}$ is a polynomial ring over $K$
in indeterminates $\tau y_{i}$ such that $\ord_{\Delta}\tau\leq r$
and $\ord_{\sigma}\tau\leq s$. Let $P_{rs} = P\bigcap R_{rs}$,
$P^{\ast}_{rs} = P^{\ast}\bigcap R_{rs}$, and let $L$, $L^{\ast}$,
$L_{rs}$ and $L^{\ast}_{rs}$ denote the quotient fields of the
integral domains $R/P$, $R/P^{\ast}$, $R_{rs}/P_{rs}$ and
$R_{rs}/P^{\ast}_{rs}$, respectively. If $\eta_{i}$ denotes the
canonical image of $y_{i}$ in $R_{rs}/P^{\ast}_{rs}$, then
$L^{\ast}$ is a $\Delta$-$\sigma$-field extension of $K$, $L^{\ast}
= K\langle \eta_{1},\dots,\eta_{n}\rangle$, and $L^{\ast}_{rs} =
K(\{\tau\eta_{i}\,|\,\tau\in T(r, s),\, 1\leq i\leq n\})$.

\medskip

The following statement is an analog of the theorem on the dimension
polynomial of an inversive difference-differential field extension
proved in \citep{Le13}.

\begin{theorem}
With the above notation, there exists a numerical polynomial
$\phi_{P^{\ast}}(t_{1}, t_{2})\in\mathbb{Q}[t_{1}, t_{2}]$ such that

{\em (i)} \,$\phi_{P^{\ast}}(r, s) = \trdeg_{K}L^{\ast}_{rs}$ for
all sufficiently large pairs $(r, s)\in\mathbb{N}^{2}$.

{\em (ii)} \, The polynomial $\phi_{P^{\ast}}(t_{1}, t_{2})$ is
linear with respect to $t_{2}$ and $\deg_{t_{1}}\phi_{P^{\ast}}\leq
m$, so this polynomial can be written as
$$\phi_{P^{\ast}}(t_{1}, t_{2}) =
\phi^{(1)}_{P^{\ast}}(t_{1})t_{2} + \phi^{(2)}_{P^{\ast}}(t_{1})$$
where $\phi^{(1)}_{P^{\ast}}(t_{1})$ and
$\phi^{(2)}_{P^{\ast}}(t_{1})$ are numerical polynomials in one
variable that, in turn, can be written as
$$\phi^{(1)}_{P^{\ast}}(t_{1}) =
\sum_{i=0}^{m}a_{i}{{t_{1}+i}\choose{i}}\,\,\,
\text{and}\,\,\,\phi^{(2)}_{P^{\ast}}(t_{1}) =
\sum_{i=0}^{m}b_{i}{{t_{1}+i}\choose{i}}$$ with $a_{i},\, b_{i}\in
\mathbb{Z}$ ($1\leq i\leq m$). Furthermore, $a_{m} =
\Delta$-$\sigma$-$\trdeg_{K}L^{\ast}$.
\end{theorem}

\begin{proof}
Let $\mathcal{A} = \{A_{1},\dots, A_{p}\}$ be a characteristic set
of the $\Delta$-$\sigma$-ideal $P^{\ast}$ and for any $r,
s\in\mathbb{N}$, let

\medskip

$U_{rs} = \{u\in TY\,|\,\ord_{\Delta}u\leq r,\, \ord_{\sigma}u\leq
s$  and either $u$ is not a multiple of any $v_{A_{i}}$ or $u$ is a
multiple of some $\sigma$-leader of an element of $\mathcal{A}$ and
for every $\tau\in T, A\in\mathcal{A}$ such that $u = \tau v_{A}$,
one has $\ord_{\Delta}(\tau u_{A}) > r\}$.

\medskip

Using our concept of an autoreduced and the arguments of the proof
of Theorem 6 in \citep[Ch. II]{Ko73}, we obtain that the set
$U_{rs}(\eta)  = \{u(\eta)\,|\,u\in U_{rs}\}$ is a transcendence
basis of $L^{\ast}_{rs}$ over $K$. In order to evaluate the number
of elements of $U_{rs}$ (and therefore, $\trdeg_{K}L^{\ast}_{rs}$),
let us consider the sets

\smallskip

$U'_{rs} = \{u\in TY\,|\,\ord_{\Delta}u\leq r,\, \ord_{\sigma}u\leq
s$  and $u$ is not a multiple of any $v_{A_{i}}\}$

\smallskip

\noindent and

\smallskip

$U''_{r, s} = \{u\in TY\,|\,\ord_{\Delta}u\leq r,\,
\ord_{\sigma}u\leq s$ and there exist $A\in\mathcal{A}$ such that $u
= \tau v_{A}$ and $\ord_{\Delta}(\tau u_{A}) > r\}$. Clearly,
$U'_{rs}\bigcap U''_{r, s} = \emptyset$ and $U_{r, s} = U'_{r,
s}\bigcup U''_{r, s}$.

\smallskip

By Theorem 2.4, there exists a numerical polynomial in two variables
$\phi^{(1)}(t_{1}, t_{2})$ such that $\phi^{(1)}(r, s) = \Card U'_{r
s}$ for all sufficiently large $(r, s)\in\mathbb{N}^{2}$,
$\deg_{t_{1}}\phi^{(1)}\leq m$, and $\deg_{t_{2}}\phi^{(1)}\leq 1$.
Furthermore, using the combinatorial principle of inclusion and
exclusion (repeating the arguments of the proof of Theorem 5.1 of
\citep{Le00} considered in the case of one translation), we obtain
that there exists a bivariate numerical polynomial
$\phi^{(2)}(t_{1}, t_{2})$ such that $\phi^{(2)}(r, s) = \Card
U''_{r s}$ for all sufficiently large $(r, s)\in\mathbb{N}^{2}$ and
$\phi^{(2)}(t_{1}, t_{2})$ is an alternating sum of bivariate
numerical polynomials of subsets of $\mathbb{N}^{m+1}$ described in
section 2. Each such a polynomial can be represented in the form
(2.1), so $\deg_{t_{1}}\phi^{(2)}\leq m$ and
$\deg_{t_{2}}\phi^{(2)}\leq 1$. Clearly the polynomial
$\phi_{P^{\ast}}(t_{1}, t_{2}) = \phi^{(1)}(t_{1}, t_{2}) +
\phi^{(2)}(t_{1}, t_{2})$ satisfies all conditions of the theorem.
The fact that $a_{m} = \Delta$-$\sigma$-$\trdeg_{K}L^{\ast}$ can be
established in the same way as in the last part of the proof of
Theorem 3.1 in \citep{Le13}.
\end{proof}

Note that in the case when $\sigma$ is an automorphism of $K$, the
statement of the last theorem was proved in \citep{Le00} with the
use of a theorem on the multivariate dimension polynomial of a
difference-differential module and properties of modules of K\"ahler
differentials.

The following theorem is the main result of the paper.

\begin{theorem}
With the notation introduced at the beginning of this section, there
exists a bivariate numerical polynomial $\psi_{P}(t_{1}, t_{2})$
such that

{\em (i)} \,$\psi_{P}(r, s) = \trdeg_{K}L_{rs}$ for all sufficiently
large pairs $(r, s)\in\mathbb{N}^{2}$.

{\em (ii)} \, The polynomial $\psi_{P}(t_{1}, t_{2})$ is linear with
respect to $t_{2}$ and $\deg_{t_{1}}\psi_{P}\leq m$, so it can be
written as
$$\psi_{P}(t_{1}, t_{2}) =
\psi^{(1)}_{P}(t_{1})t_{2} + \psi^{(2)}_{P}(t_{1})$$ where
$\psi^{(1)}_{P}(t_{1})$ and $\psi^{(2)}_{P}(t_{1})$ are numerical
polynomials in one variable.
\end{theorem}

\begin{proof}
We start with the proof for the case $\Delta = \emptyset$. In this
case we will use the above notation and conventions just replacing
the prefix $\Delta$-$\sigma$- by $\sigma$- (and
''difference-differential'' by ''difference''). Let $\mathcal{A} =
\{A_{1},\dots, A_{p}\}$ be a characteristic set of the
$\sigma$-ideal $P^{\ast}$ (the reflexive closure of the prime
$\sigma$-ideal $P$ of the ring of $\sigma$-polynomials $R =
K\{y_{1},\dots, y_{n}\}$) and let $v_{j}$ denote the $\sigma$-leader
of $A_{j}$ ($j=1,\dots, p$). Let $\eta_{i} = y_{i}+P$ ($1\leq i\leq
n$), $L = K(\{\sigma^{k}\eta_{i}\,|\,k\in\mathbb{N},\, 1\leq i\leq
n\})$ (the quotient field of $R/P$) and $L_{s} =
K(\{\sigma^{k}\eta_{i}\,|\,0\leq k\leq s,\, 1\leq i\leq n\})$.

For every $j=1,\dots, p$, let $s_{j}$ be the smallest nonnegative
integer such that $\sigma^{s_{j}}(A_{j})\in P$. Furthermore, let
$$V = \{u\in TY\,|\,u\neq\sigma^{i}v_{j}\,\,\, \text{for any}\,\,\,
i\in\mathbb{N},\, 1\leq j\leq p\},$$ $V_{r} = \{v\in
V\,|\,\ord_{\sigma}v\leq r\}$ ($r\in\mathbb{N}$), $V(\eta) =
\{u(\eta)\,|\,u\in V\},$ $W = \{\sigma^{k}v_{j}\,|\,1\leq j\leq p,
\,0\leq k\leq s_{j}-1\}$,\, and\,  $W(\eta) = \{u(\eta)\,|\,u\in
W\}.$

\medskip

It is easy to see that the set $V(\eta)$ is algebraically
independent over $K$. Indeed, suppose there exist $v_{1},\dots,
v_{l}\in V$ and a polynomial $f(X_{1},\dots, X_{l})$ in $l$
variables with coefficients in the field $K$ such that
$f(v_{1}(\eta),\dots, v_{l}(\eta)) = 0$. Then $f(v_{1},\dots,
v_{l})\in P\subseteq P^{\ast}$ and $f(v_{1},\dots, v_{l})$ is
reduced with respect to the characteristic set $\mathcal{A}$ (this
$\sigma$-polynomial does not contain any transforms of the leaders
of elements of $\mathcal{A}$). Therefore, $f = 0$, so the set
$V(\eta)$ is algebraically independent over $K$.

\smallskip

Now we notice that every element of the field $L$ is algebraic over
its subfield $K\left(V(\eta)\bigcup W(\eta)\right)$.

Indeed, since $L = K(V(\eta)\bigcup
W(\eta)\bigcup\{\sigma^{k}v_{j}(\eta)\,|\,1\leq j\leq p, k\geq
s_{i}\})$, it is sufficient to prove that every element
$\sigma^{k}v_{j}(\eta)$ with $k\geq s_{j}$ ($1\leq i\leq p$) is
algebraic over $K\left(V(\eta)\bigcup W(\eta)\right)$.

Since $\sigma^{s_{j}}A_{j}\in P$, we have $\sigma^{s_{j}}A_{j}(\eta)
= 0$, hence $\sigma^{k}A_{j}(\eta) = 0$ for all $k\geq s_{j}$. If
one writes $A_{j}$ as a polynomial in $v_{j}$,
$$A_{j} = I^{(j)}_{q_{j}}v_{j}^{q_{j}} +
I^{(j)}_{q_{j}-1}v_{j}^{q_{j}-1}+\dots + I^{(j)}_{0}$$
($I^{(j)}_{q_{j}},\, \dots, I^{(j)}_{0}$ do not contain $v_{j}$) and
$k\geq s_{j}$, then

\medskip

$\sigma^{k}A_{j}(\eta) =
\left(\sigma^{k}I^{(j)}_{q_{j}}(\eta)\right)v_{j}(\eta)^{q_{j}} +
\left(\sigma^{k}I^{(j)}_{q_{j}-1}(\eta)\right)v_{j}(\eta)^{q_{j}-1}
+\dots + \sigma^{k}I^{(j)}_{0}(\eta) = 0.$

\medskip

\noindent Note that $I^{(j)}_{q_{j}}\notin P^{\ast}$, since
$I^{(j)}_{q_{j}}$ is an initial of an element of the characteristic
set of $P^{\ast}$ and therefore is reduced with respect to this set
(by Proposition 3.11, the inclusion $I^{(j)}_{q_{j}}\in P^{\ast}$
would imply $I^{(j)}_{q_{j}}=0$). Since the $\sigma$-ideal
$P^{\ast}$ is reflexive, $\sigma^{k}I^{(j)}_{q_{j}}\notin P^{\ast}$,
hence $\sigma^{k}I^{(j)}_{q_{j}}(\eta)\neq 0$. It follows that
$\sigma^{k}v_{j}(\eta)$ is algebraic over the field
$K\left(V(\eta)\bigcup W(\eta)\bigcup\{u(\eta)\,|\,u <_{\sigma}
\sigma^{k}u_{i}\}\right)$ (the term ordering $<_{\sigma}$ was
defined at the beginning of section 3).  By induction on the
well-ordered (with respect to $<_{\sigma}$) set of terms $TY$ we
obtain that all elements $\sigma^{k}v_{j}(\eta)$ ($1\leq j\leq p$,
$k\in \mathbb{N}$) are algebraic over $K\left(V(\eta)\bigcup
W(\eta)\right)$, so $L$ is algebraic over this field as well.

\smallskip

Let $\{w_{1},\dots, w_{q}\}$ be a maximal subset of $W$ such that
the set $\{w_{1}(\eta),\dots, w_{q}(\eta)\}$ is algebraically
independent over $K\left(V(\eta)\right)$. Then
$V(\eta)\bigcup\{w_{1}(\eta),\dots, w_{q}(\eta)\}$ is a
transcendence basis of the field $L$ over $K$. Furthermore, since
the set $W(\eta)$ is finite, there exists $r_{0}\in\mathbb{N}$ such
that

\noindent(i) $w_{1},\dots, w_{q}\in R_{r_{0}}$;

\noindent(ii) $r_{0}\geq \max\{\ord_{\sigma}v_{j}+s_{j}\,|\,1\leq
j\leq p\}$;

\noindent(iii) Every element of $W(\eta)$ is algebraic over the
field $K\left(V_{r_{0}}(\eta)\bigcup\{w_{1}(\eta),\dots,
w_{q}(\eta)\}\right)$.

\medskip

Let $r\geq r_{0}$, $R_{r} = K[\{\sigma^{k}y_{i}\,|\,1\leq i\leq n,\,
0\leq k\leq r\}]$, and $P_{r}= P\bigcap R_{r}$. Let $L_{r}$  denote
the quotient field of the integral domain $R_{r}/P_{r}$ and
$\zeta^{(r)}_{i} = y_{i}+P_{r}\in R_{r}/P_{r}\subseteq L_{r}$
($1\leq i\leq n$). Furthermore, let $\zeta^{(r)} =
\{\zeta^{(1)}_{1},\dots, \zeta^{(1)}_{n}\}$,  and
$V_{r}(\zeta^{(r)}) = \{v(\zeta^{(r)})\,|\,v\in V_{r}\}$. We are
going to prove that $$B_{r} =
V_{r}(\zeta^{(r)})\bigcup\{w_{1}(\zeta^{(r)}),\dots,
w_{q}(\zeta^{(r)})\}$$ is a transcendence basis of $L_{r}$ over $K$.

Repeating the arguments of the proof of Theorem 4.1 (applied to the
set $V_{r}(\zeta^{(r)})$ instead of $V(\eta)$) we obtain that
$V_{r}(\zeta^{(r)})$  is algebraically independent over $K$.  Let us
show that  the elements $w_{1}(\zeta^{(r)}),\dots,
w_{q}(\zeta^{(r)})$ are algebraically independent over the field
$K(V_{r}(\zeta^{(r)}))$. Suppose that $g(w_{1}(\zeta^{(r)}),\dots,
w_{q}(\zeta^{(r)})) = 0$ for some polynomial $g$ in $q$
indeterminates with coefficients in
$K\left(V_{r}(\zeta^{(r)})\right)$. Then there exist elements
$z_{1},\dots, z_{d}\in V_{r}$ such that all coefficients of $f$ lie
in $K(z_{1}(\zeta^{(r)}),\dots, z_{d}(\zeta^{(r)})\,)$. Multiplying
$f$ by the common denominator of these coefficients we obtain a
nonzero polynomial $g$ in $d+q$ indeterminates such that
$$g(z_{1}(\zeta^{(r)}),\dots, z_{d}(\zeta^{(r)}),
w_{1}(\zeta^{(r)}),\dots, w_{q}(\zeta^{(r)})) = 0,$$ hence
$g(z_{1},\dots, z_{d}, w_{1},\dots, w_{q})\in P_{r}\subseteq P$.
Considering the image of $g$ under the natural homomorphism
$R\rightarrow R/P\subseteq L$ we obtain that $g(z_{1}(\eta),\dots,
z_{d}(\eta), w_{1}(\eta),\dots, w_{q}(\eta)) = 0$ where
$z_{i}(\eta)\in K\left(V(\eta)\right)$, $1\leq i\leq d$. Since
$V(\eta)\bigcup\{w_{1}(\eta),\dots, w_{q}(\eta)\}$ is a
transcendence basis of $L/K$, $g = 0$, a contradiction. Therefore,
the elements $w_{1}(\zeta^{(r)}),\dots, w_{q}(\zeta^{(r)})$ are
algebraically independent over $K\left(V_{r}(\zeta^{(r)})\right)$,
so that the set $B_{r}$ is algebraically independent over $K$.

\medskip

Now let $r\geq r_{0}$ and let $u = \sigma^{l}y_{i}$ where $0\leq
l\leq r$ ($1\leq i\leq n$). If $u$ is not a transform of any $v_{k}$
($1\leq k\leq p$), then $u\in V_{r}$ and $u(\zeta)\in V_{r}(\zeta)$.

If $u = \sigma^{j}v_{k}$ where $0\leq j\leq s_{k}-1$ (in this case
$\ord_{\sigma}u < r$), then $u\in W$, hence the element $u(\eta)$ is
algebraic over the field
$K\left(V_{r}(\eta)\bigcup\{w_{1}(\eta),\dots,
w_{q}(\eta)\}\right)$. As above, we obtain the existence of a
nonzero polynomial $h$ in $d+q+1$ variables ($d\in \mathbb{N}$) and
elements $z_{1},\dots, z_{d}\in V_{r}$ such that $h(z_{1},\dots,
z_{d}, w_{1},\dots, w_{q}, u)\in P_{r}$. Then
$h(z_{1}(\zeta^{(r)}),\dots, z_{d}(\zeta^{(r)}),
w_{1}(\zeta^{(r)}),\dots, w_{q}(\zeta^{(r)}), u(\zeta^{(r)})\,) = 0$
hence $u(\zeta^{(r)})$ is algebraic over the field $K(B_{r})$.

Suppose that $u = \sigma^{j}v_{k}$ where $s_{k}\leq j\leq r-\ord
v_{k}$ ($1\leq k\leq p$). Then $\sigma^{j}A_{k}\in P_{r}$, hence
$\sigma^{j}A_{k}(\zeta^{(r)}) = 0$. If one writes $A_{k}$ as a
polynomial of $v_{k}$,
$$A_{k} = I_{kd_{k}}v_{k}^{d_{k}} +\dots + I_{k1}v_{k} + I_{k0}$$
($I_{ij}$ do not contain $v_{k}$ and all terms in $I_{ij}$ are lower
than $v_{k}$ with respect to $<_{\sigma}$), then
$\sigma^{j}I_{kd_{k}}\notin P^{\ast}$, since $I_{kd_{k}}$ is the
initial of an element of a characteristic set of $P^{\ast}$ and the
ideal $P^{\ast}$  is reflexive. Therefore,
$\sigma^{j}I_{kd_{k}}(\zeta^{(r)})\neq 0$, so the equality
$\sigma^{j}A_{k}(\zeta^{(r)}) = 0$ shows that $u(\zeta^{(r)}) =
\sigma^{j}v_{k}(\zeta^{(r)})$ is algebraic over the field
$K\left(V_{r}(\zeta^{(r)})\bigcup\{v(\zeta^{(r)})\,|\,v\in TY,\, v
<_{\sigma} u\}\right)$.

Using the induction on the well-ordered (with respect to
$<_{\sigma}$) set $TY$ we obtain that $u(\zeta^{(r)})$ is algebraic
over the field
$K\left(V_{r}(\zeta^{(r)})\bigcup\{\sigma^{j}v_{k}(\zeta^{(r)})\,|\,
1\leq k\leq d, \, 0\leq j\leq s_{k}-1\}\right)$ and this field, as
we have seen, is algebraic over $K(B_{r})$. It follows that
$u(\zeta^{(r)})$ is algebraic over $K(B_{r})$ for every term $u$
with $\ord_{\sigma}u\leq r$. Therefore, $B_{r}$ is a transcendence
basis of $L_{r}$ over $K$.

As we have seen, $\Card B_{r} = \Card V_{r} + q$. By \citep[Ch. 0,
Lemma 16]{Ko73}, there exists a numerical polynomial $\omega(t)$ in
one variable $t$ such that $\omega(r) = \Card V_{r}$ for all
sufficiently large $r\in\mathbb{Z}$. Therefore, the numerical
polynomial $\psi_{P}(t) = \omega(t) + q$ satisfies the desired
conditions: $\psi_{P}(r) = \trdeg_{K}L_{r}$ for all sufficiently
large $r\in\mathbb{Z}$.

\smallskip

Now we are going to complete the proof of the theorem considering
the case when $\Card\Delta = m > 0$. In this case, the field
$L_{rs}$ can be treated as the subfield
$K(\{\theta\sigma^{j}\xi_{i}\,|\,\theta\in\Theta(r), 0\leq j\leq s,
1\leq i\leq n\})$ of the differential ($\Delta$-) overfield
$K\langle\{\sigma^{j}\xi_{i}\,|\,0\leq j\leq s, 1\leq i\leq
n\}\rangle_{\Delta}$ of $K$. ($\xi_{i}$ is the canonical image of
$y_{i}$ in $R_{rs}/P_{rs}$; the index $\Delta$ indicates that we
consider a differential, not a difference-differential, field
extension.)

By the Kolchin's theorem on differential dimension polynomial (see
\citep[Ch.II, Theorem 6]{Ko73}\,), for any $s\in\mathbb{N}$, there
exists a numerical polynomial $\chi_{s}(t) =
\sum_{i=0}^{m}a_{i}(s)\D{t+i\choose i}$ in one variable $t$ such
that $\chi_{s}(r) = \trdeg_{K}L_{rs}$ for all sufficiently large
$r\in \mathbb{N}$ and $a_{i}(s)\in\mathbb{Z}$ ($0\leq i\leq m$).

On the other hand, the first part of the proof (with the use of the
finite set of $\sigma$-indeterminates
$\{\theta(r)y_{i}\,|\,\theta\in\Theta(r),\, 1\leq i\leq n\}$ instead
of $\{y_{1},\dots, y_{n}\}$) shows that $\trdeg_{K}L_{rs} = \Card
V_{rs} + \lambda(r)$ where

\medskip

$V_{rs} = \{u=\tau y_{i}\in TY\,|\,\tau\in T(r, s)$ and $u\neq \tau'
v_{j}$ for any $\tau'\in T,\, 1\leq j\leq p\}$.

\medskip

\noindent($v_{j}$ denotes the $\sigma$-leader of the element $A_{j}$
of a characteristic set $\mathcal{A} = \{A_{1},\dots, A_{p}\}$ of
the reflexive closure $P^{\ast}$ of $P$.) Since the set $W$ in the
first part of the proof is finite and depends only on the
$\sigma$-orders of terms of $A_{j}$, $1\leq j\leq p$, the number of
elements of the corresponding set in the general case depends only
on $r$; we have denoted it by $\lambda(r)$. (More precisely, let
$R^{(r)} = K\langle\{\theta y_{i}\,|\,\theta\in\Theta(r),\, 1\leq
i\leq n\}\rangle_{\sigma}$ be the ring of difference ($\sigma$-)
polynomials in difference indeterminates $\theta y_{i}$
($\theta\in\Theta(r),\, 1\leq i\leq n$) over $K$, $\mathcal{A}(r) =
\{A_{1}^{(r)},\dots, A_{p(r)}^{(r)}\}$  a characteristic set of the
difference ideal $P^{\ast}\bigcap R^{(r)}$ of $R^{(r)}$ and
$v_{j}(r)$ the leader of $A_{j}^{(r)}$ (treated as a difference
polynomial). If $s_{j}(r)$ is the smallest nonnegative integer such
that $\sigma^{s_{j}(r)}(A_{j}^{(r)})\in P\bigcap R^{(r)}$, then the
set $W(r)$ that correspond to the set $W$ in the first part of the
proof is of the form $W(r) = \{\sigma^{k}v_{j}(r)\,|\,1\leq j\leq
p(r), 0\leq k\leq s_{j}(r)\}$.)

By Theorem 2.4, there exist $r_{0}, s_{0}\in\mathbb{N}$ and  a
bivariate numerical polynomial $\omega(t_{1}, t_{2})$ such that
$\omega(r, s) = \Card V_{rs}$ for all $r\geq r_{0},\, s\geq s_{0}$,
$\deg_{t_{1}}\omega\leq m$ and $\deg_{t_{2}}\omega\leq 1$. (This
polynomial can be computed with the use of formula (2.2).)
Therefore, $\trdeg_{K}L_{rs} = \omega(r, s) + \lambda(r)$ for all
$r\geq r_{0},\, s\geq s_{0}$. At the same time, we have seen that
$\trdeg_{K}L_{rs_{0}} = \chi_{s_{0}}(r) =
\sum_{i=0}^{m}a_{i}(s_{0})\D{r+i\choose i}$ for all sufficiently
large $r\in \mathbb{N}$ ($a_{i}(s_{0})\in\mathbb{Z}$). It follows
that $\lambda(r)$ is a polynomial of $r$ for all sufficiently large
$r\in\mathbb{N}$, say, for all $r\geq r_{1}$. Therefore, for any
$s\geq s_{0},\, r\geq\max\{r_{0}, r_{1}\}$, $\trdeg_{K}L_{rs} =
\omega(r, s) + \lambda(r)$ is expressed as a bivariate numerical
polynomial in $r$ and $s$.
\end{proof}

\begin{definition}
The numerical polynomial $\psi_{P}(t_{1}, t_{2})$ whose existence is
established by Theorem 4.2 is called the $\Delta$-$\sigma$-dimension
polynomial of the $\Delta$-$\sigma$-ideal $P$.
\end{definition}

The proof of the last theorem (as well as the proof of Theorem 4.1)
shows that the main step in the computation of a
$\Delta$-$\sigma$-dimension polynomial is the construction of a
characteristic set in the sense of section 3. It can be realized by
the corresponding generalization of the Ritt-Kolchin algorithm
described in \citep[Section 5.5]{KLMP99}, but the development and
implementation of such a generalization is the subject of future
research.

\medskip

\noindent{\bf Remark.} The result of Theorem 4.2 raises the question
whether a bivariate dimension polynomial of similar kind can be
assigned to a non-reflexive difference-differential polynomial ideal
in the case of several translations, that is, when the endomorphism
$\sigma$ in the conditions of Theorem 4.2 is replaced with a finite
set of endomorphisms $\Sigma = \{\sigma_{1},\dots, \sigma_{h}\}$
whose elements commute between themselves and commute with basic
derivations. (In this case,
$\ord\delta_{1}^{k_{1}}\dots\delta_{m}^{k_{m}}\sigma_{1}^{l_{1}}\dots\sigma_{h}^{l_{h}}
= \sum_{i=1}^{m}k_{i} + \sum_{j=1}^{h}l_{j}$).) Unfortunately, our
proof of Theorem 4.2 does not work in this case. Indeed, if there
are several translations and $\tau_{1},\dots, \tau_{p}$ are power
products of elements of $\Sigma$ such that $\tau_{i}(A_{i})\in P$
(we use the notation of the first part of the proof of Theorem 4.2;
$\tau_{i}$ are analogs of $\sigma^{s_{i}}$ in the proof), then the
set $W$ of all terms that are multiples of some $v_{i}$ but not
multiples of any $\tau_{i}(v_{i})$ is, generally speaking, infinite.
Thus, the arguments of our proof cannot be applied. However, the
author conjectures that the analog of Theorem 4.2 holds for
difference-differential polynomial rings with several translations.

\begin{corollary}
With the notation of Theorem 4.2, let $P$ be a prime
$\Delta$-$\sigma$-ideal of the ring of $\Delta$-$\sigma$-polynomials
$R = K\{y_{1},\dots, y_{n}\}$ and let $P^{\ast}$ be the reflexive
closure of $P$ in $R$. Then there exists  $k\in\mathbb{N}$ such that
$P^{\ast} = \sigma^{-k}(P)$.
\end{corollary}

\begin{proof}
Let $\mathcal{A}$ be a characteristic set of $P^{\ast}$. Since the
set $\mathcal{A}$ is finite, there exists $k\in\mathbb{N}$ such that
$\sigma^{k}(\mathcal{A})\subseteq P$. Clearly, $\sigma^{-k}(P)$ is a
prime $\Delta$-$\sigma$-ideal of $R$, its reflexive closure is
$P^{\ast}$, and $\mathcal{A}$ is a characteristic set of
$\sigma^{-k}(P)$. Applying the first part of the proof of Theorem
4.2 to $\sigma^{-k}(P)$ and its reflexive closure $P^{\ast}$ (in
this case all $s_{j}$ in the proof are zeros), we obtain that the
$\Delta$-$\sigma$-dimension polynomials of the ideals
$\sigma^{-k}(P)$ and $P^{\ast}$ are equal. Therefore,
$\sigma^{-k}(P)\bigcap R_{rs} = P^{\ast}\bigcap R_{rs}$ for all
sufficiently large $(r, s)\in\mathbb{N}^{2}$ (as before, $R_{rs} =
K[\{\tau y_{i}\,|\,\tau\in T(r, s), 1\leq i\leq n\}]$), hence
$P^{\ast} = \sigma^{-k}(P)$.
\end{proof}

\noindent{\bf Remark.} I would like to thank the anonymous reviewer
of the paper for the following alternative proof of the last
statement.

As above, let $\mathcal{A}$ be a characteristic set of $P^{\ast}$
and let $k$ be a positive integer such that
$\sigma^{k}(\mathcal{A})\subseteq P$. By Propositions 3.7 and 3.11,
if $f\in P^{\ast}$, then there exists $H\in R$ such that $H$ is a
finite product of transforms of separants and leading coefficients
of elements of $\mathcal{A}$ and $Hf\in[\mathcal{A}]$. Then
$\sigma^{k}(Hf) =
\sigma^{k}(H)\sigma^{k}(f)\in\sigma^{k}([\mathcal{A}])\subseteq P$.
Since the separants and leading coefficients of elements of
$\mathcal{A}$ are reduced with respect to $\mathcal{A}$ and the
ideal $P^{\ast}$ is prime, $H\notin P^{\ast}$ (see Proposition
3.11). Therefore, $\sigma^{k}(H)\notin P$, hence $\sigma^{k}(f)\in
P$, so $f\in\sigma^{-k}(P)$. Thus, $P^{\ast} = \sigma^{-k}(P)$.

\smallskip

This proof has the advantage that it can be extended to the proof of
the corresponding statement for the case of several translations.
Indeed, let $\Sigma = \{\sigma_{1},\dots, \sigma_{h}\}$ be the basic
set of translations of a difference-differential
($\Delta$-$\Sigma$-) field $K$, $\Gamma$ the free commutative
semigroup of all power products $\sigma_{1}^{l_{1}}\dots
\sigma_{h}^{l_{h}}$ ($l_{j}\in\mathbb{N}$) and $R = K\{y_{1},\dots,
y_{n}\}$ the ring of $\Delta$-$\Sigma$-polynomials over $K$ (as a
ring, $R$ is a polynomial ring in the set of indeterminates $\lambda
y_{i}$ ($1\leq i\leq n$) where $\lambda$ is a power product of
elements of $\Delta\bigcup\Sigma$ with nonnegative integer
exponents, see \citep[Sect. 3.5]{KLMP99} for details). If $P$ is a
prime $\Delta$-$\Sigma$-ideal of $R$ and $P^{\ast}$ the reflexive
closure of $P$ (that is, the set of all
$\Delta$-$\Sigma$-polynomials $f\in R$ such that $\gamma_{f}(f)\in
P$ for some $\gamma_{f}\in\Gamma$), then the arguments of the last
proof show that $P^{\ast} = \gamma^{-1}(P)$ for some
$\gamma\in\Gamma$.

\begin{corollary}
With the above notation, let $K$ be a $\Delta$-$\sigma$-filed and $R
= K\{y_{1},\dots, y_{n}\}$ the ring of $\Delta$-$\sigma$-polynomials
over $K$. Then the ring $R$ satisfies the ascending chain condition
for prime $\Delta$-$\sigma$-ideals. In particular, the ring of
difference polynomials over an ordinary difference field satisfies
the ascending chain condition for prime difference ideals.
\end{corollary}

\begin{proof}
Let $P_{1}\subseteq P_{2}\subseteq\dots$ be an ascending chain of
prime $\sigma$-ideals of $R$. For any $r\in\mathbb{N}$, let $R_{rs}$
denote the polynomial ring $K[\{\tau y_{i}\,|\,\tau\in T(r, s),
1\leq i\leq n\}]$, let $P_{irs} = P_{i}\bigcap R_{rs}$, and let
$L_{irs}$ denote the quotient $\sigma$-field of $R_{rs}/P_{irs}$. By
Theorem 4.2, for every $i = 1, 2,\dots$, there exists a bivariate
numerical polynomials $\psi_{P_{i}}(t_{1}, t_{2})$ such that
$\psi_{P_{i}}(r, s) =\trdeg_{K}L_{irs}$ for all sufficiently large
$(r, s)\in\mathbb{N}^{2}$ (for all $(r, s)\in\mathbb{N}^{2}$ with
$r\geq r_{0}$ and $s\geq s_{0}$ where $r_{0}$ and $s_{0}$ depend on
$i$).

Let $\mathcal{W}$ denote the set of all bivariate numerical
polynomials that can serve as dimension polynomials of prime
$\Delta$-$\sigma$-ideals of $R$. Let $\succcurlyeq$ be a partial
order on $\mathcal{W}$ such that $\phi_{1}(t_{1}, t_{2})\succcurlyeq
\phi_{2}(t_{1}, t_{2})$ if and only if $\phi_{1}(r,
s)\geq\phi_{2}(r, s)$ for all sufficiently large $(r,
s)\in\mathbb{N}^{2}$. Then $$\psi_{P_{1}}(t_{1}, t_{2})\succcurlyeq
\psi_{P_{2}}(t_{1}, t_{2})\succcurlyeq\dots.$$ By \citep[Theorem
2.4.26]{KLMP99}, the set $\mathcal{W}$ satisfies the descending
chain condition with respect to $\succcurlyeq$. Therefore, there
exists $j\geq 1$ such that $$\psi_{P_{j}}(t_{1}, t_{2}) =
\psi_{P_{j+1}}(t_{1}, t_{2}) = \dots.$$ It follows that for all
sufficiently large $(r, s)\in\mathbb{N}^{2}$ and for any $k\geq 1$,
the prime ideals $P_{jrs}$ and $P_{j+k,rs}$ have the same height in
the ring $R_{rs}$. Since $P_{jrs}\subseteq P_{j+k,rs}$, we obtain
that $P_{jrs} = P_{j+k,rs}$ for all sufficiently large $(r,
s)\in\mathbb{N}^{2}$, hence $P_{j} = P_{k+j}$, so $P_{j} = P_{j+1}
=\dots$.
\end{proof}

The following illustrating example uses the notation of the proofs
of Theorems 4.1 and 4.2.

\begin{example}
{\em Let $K$ be a difference-differential ($\Delta$-$\sigma$-) field
with two basic derivations, $\Delta = \{\delta_{1}, \delta_{2}\}$,
and one basic endomorphism $\sigma$. Let $K\{y\}$ be the ring of
$\Delta$-$\sigma$-polynomials in one $\Delta$-$\sigma$-indeterminate
$y$ over $K$ and let $P$ be a linear (and therefore prime)
$\Delta$-$\sigma$-ideal of $K\{y\}$ generated by the
$\Delta$-$\sigma$-polynomial $A = \sigma^{2}y +
\sigma\delta_{1}^{2}y + \sigma\delta_{2}^{2}y$ (that is, $P = [A]$).
Then $P^{\ast} = [B]$, where $B = \sigma y + \delta_{1}^{2}y +
\delta_{2}^{2}y$, and Proposition 3.13 shows that $\{B\}$ is a
characteristic set of the $\Delta$-$\sigma$-ideal $P^{\ast}$. With
the notation of the proof of Theorem 4.1, we have $U'_{rs} = \{u\in
TY\,|\,\ord_{\Delta}u\leq r,\, \ord_{\sigma}u\leq s$ and $u$ is not
a multiple of $\sigma y\}$ and $U''_{rs} = \{u\in
TY\,|\,\ord_{\Delta}u\leq r,\, \ord_{\sigma}u\leq s$ and there is
$\tau\in T$ such that $u = \tau(\sigma y)$ and
$\ord_{\Delta}(\tau\delta_{1}^{2}) > r\}$. Then $\Card U'_{rs} =
\Card\{\delta_{1}^{i}\delta_{2}^{j}y\,|\,i+j\leq r\} =
\D{{r+2}\choose 2}$ and

\medskip

\noindent$\Card U''_{rs} =
\Card\{\sigma^{i}\delta_{1}^{j}\delta_{2}^{k}y\,|\,1\leq i\leq s,\,
r-2 < j+k\leq r\} = s\left(\D{{r+2}\choose 2} -  \D{{r+2-2}\choose
2}\right) = (2r+1)s.$

Since $\sigma B\in P$, the proof of Theorem 4.2 shows that if
$\psi_{P}(t_{1}, t_{2})$ is the $\Delta$-$\sigma$-dimension
polynomial of the $\Delta$-$\sigma$-ideal $P$, then $$\psi(r, s) =
\Card U'_{rs} + \Card U''_{rs} +
\Card\{\sigma\delta_{1}^{i}\delta_{2}^{j}y\,|\,i+j\leq r-2\}$$ for
all sufficiently large $(r, s)\in\mathbb{N}^{2}$. It follows that\\
$\psi_{P}(t_{1}, t_{2}) = (2t_{1}+1)t_{2} + \D{{t_{1}+2}\choose 2} +
\D{{t_{1}}\choose 2}$, that is $$\psi_{P}(t_{1}, t_{2}) =
(2t_{1}+1)t_{2} + t_{1}^{2} + t_{1} + 1.$$}
\end{example}

We conclude with a brief discussion of the connection between the
$\Delta$-$\sigma$-dimension polynomial and the concept of strength
of a system of difference-differential equations in the sense of A.
Einstein.

Consider a system of difference-differential equations
\begin{equation}
A_{i}(f_{1},\dots, f_{n}) = 0\hspace{0.3in}(i=1,\dots, q)
\end{equation}
with $m$ basic partial derivations and one translation $\sigma$ over
a field $K$ of functions of $m$ real variables $x_{1},\dots, x_{m}$
treated as a difference-differential field with basic set of
derivations $\Delta = \{\delta_{1},\dots, \delta_{m}\}$ and one
translation $\sigma$ where $\delta_{i} =
\partial/\partial x_{i}$ ($1\leq i\leq m$) and
$\sigma:f(\overline{x})\mapsto f(\overline{x}+\overline{h})$ is a
shift of the argument $\overline{x}=(x_{1},\dots, x_{m})$ by some
vector $\overline{h}$ in $\mathbb{R}^{m}$. ($f_{1},\dots, f_{n}$ are
unknown functions of $x_{1},\dots, x_{m}$).  We assume that system
(4.3) is algebraic, that is, all $A_{i}(y_{1},\dots, y_{n})$ are
elements of a ring of $\Delta$-$\sigma$-polynomials $K\{y_{1},\dots,
y_{n}\}$ over the functional $\Delta$-$\sigma$-field $K$.

Let us consider a sequence of nodes in $\mathbb{R}^{m}$ that begins
at some initial node $\mathcal{P}$ and goes in the direction of the
vector $\overline{h}$ with step $|\overline{h}|$. We say that {\em a
node $\mathcal{Q}$ has $\sigma$-order $i$} (with respect to
$\mathcal{P}$) if the distance between $\mathcal{Q}$ and
$\mathcal{P}$ is $i|\overline{h}|$.

Let us consider the values of the unknown functions $f_{1},\dots,
f_{n}$ and their partial derivatives of order at most $r$ at the
nodes of $\sigma$-order at most $s$ ($r$ and $s$ are positive
integers). With the notation of section 2, we can say that we
consider the values $\tau f_{i}(\mathcal{P})$ where $\tau\in T$,
$\ord_{\Delta}\tau\leq r$ and $\ord_{\sigma}\tau\leq s$.

If $f_{1},\dots, f_{n}$ should not satisfy any system of equations
(or any other condition), these values can be chosen arbitrarily.
Because of the system (and equations obtained from the equations of
the system by partial differentiations and translations in the
direction $\overline{h}$, the number of independent values of the
functions $f_{1},\dots, f_{n}$ and their partial derivatives whose
order does not exceed $r$ at the nodes of $\sigma$-order at most $s$
decreases. This number, which is a function of two variables, $r$
and $s$,  is the ''measure of strength'' of the system in the sense
of A. Einstein. We denote it by $S_{rs}$. Suppose that the
$\Delta$-$\sigma$-ideal $P$ generated in $K\{y_{1},\dots, y_{n}\}$
by the $\Delta$-$\sigma$-polynomials $A_{1},\dots, A_{q}$ is prime
(e. g., the polynomials are linear). Then we say that the system of
difference-differential equations (4.3) is prime. In this case, the
$\Delta$-$\sigma$-dimension polynomial $\psi_{P}(t_{1}, t_{2})$ has
the property that $\psi_{P}(r, s) = S_{rs}$ for all sufficiently
large $(r, s)\in \mathbb{N}^{2}$, so this dimension polynomial is
the measure of strength of the system of difference-differential
equations (4.3) in the sense of A. Einstein. An important
perspective for the use of the obtained results is the computation
of dimension polynomials (and therefore the Einstein's strength) of
differential equations with delay that arise in applications.
Examples of the corresponding computation in the differential and
inverse difference cases can be found in \citep[Ch. 6 and
9]{KLMP99}. Computations of the same kind in the non-inverse case
(based on the results of this paper) is a subject for future work.

\section{Acknowledges}

This research was supported by the NSF grant CCF-1714425 and NSA
grant H98230-18-1-0016.

\end{document}